\documentclass[12pt,reqno]{amsart}
\usepackage{amssymb}
\usepackage{epsf,graphicx}
 
\newtheorem{theorem}{Theorem}[section]
\newtheorem{proposition}[theorem]{Proposition}
\newtheorem{lemma}[theorem]{Lemma}

\theoremstyle{definition}
\newtheorem{definition}[theorem]{Definition}
\newtheorem{remark}[theorem]{Remark}
\newtheorem{conjecture}[theorem]{Conjecture}

\numberwithin{equation}{section}

\newcommand{\X}{\mathbb{X}}
\newcommand{\M}{\mathbb{M}}
\newcommand{\E}{\mathbb{E}}
\newcommand{\Q}{\mathbb{Q}}
 
\newcommand{\C}{\mathbb{C}}
\newcommand{\bP}{\mathbb{P}}
\newcommand{\SO}{\mathcal{O}}
\newcommand{\SI}{\mathcal{I}}

\begin{document}
\baselineskip=15pt

\title[Degenerations of bundle moduli]{Degenerations of bundle moduli}

\author[I. Biswas]{Indranil Biswas}

\address{School of Mathematics, Tata Institute of Fundamental
Research, 1 Homi Bhabha Road, Mumbai 400005, India}

\email{indranil@math.tifr.res.in}

 \author[J. Hurtubise]{Jacques Hurtubise}

\address{Department of Mathematics, McGill University, Burnside
Hall, 805 Sherbrooke St. W., Montreal, Que. H3A 2K6, Canada}

\email{jacques.hurtubise@mcgill.ca}



\date{}

\begin{abstract}
Over a family $\X$ of complete curves of genus $g$, which gives the degeneration of a smooth curve into 
one with nodal singularities, we build a moduli space which is the moduli space of holomorphic 
${\rm SL}(n, \mathbb C)$
bundles over the generic smooth curve $X_t$ in the family, and is a moduli 
space of bundles equipped with extra structure at the nodes for the nodal curves in the 
family. This moduli space is a quotient by $(\mathbb C^*)^s$ of a moduli space on the 
desingularisation. Taking a ``maximal'' degeneration of the curve into a nodal curve built 
from the glueing of three-pointed spheres, we obtain a degeneration of the moduli space of 
bundles into a $(\mathbb C^*)^{(3g-3)(n-1)}$-quotient of a $(2g-2)$-th power of a space 
associated to the three-pointed sphere. Via the Narasimhan-Seshadri theorem, the moduli space of 
${\rm SL}(n, \mathbb C)$
bundles on the smooth curve is a space of representations of the fundamental group into
${\rm SU}(n)$ (the ``symplectic picture''). We obtain the degenerations also in this symplectic 
context, in a way that is compatible with the holomorphic degeneration, so that our limit 
space is also a $(S^1)^{(3g-3)(n-1)}$ symplectic quotient of a $(2g-2)$-th power of a space 
associated to the three-pointed sphere.
\end{abstract}

\maketitle

\tableofcontents

\section{Introduction} 

The celebrated Verlinde formulae, \cite{V}, give the dimension of the spaces of sections of line bundles over the 
moduli spaces of vector bundles on a compact Riemann surface. There are several proofs of this, but a beautiful result 
of Jeffrey and Weitsman, in \cite{JW}, gives in our minds for ${\rm SL}(2,{\mathbb C})$ one of the most elegant 
explanations, if not quite a proof, of the result. By the Narasimhan-Seshadri theorem \cite{NS}, this moduli 
space is equivalent to the representation variety of flat ${\rm SU}(2)$ connections on the Riemann surface, and 
the paper \cite{JW} shows that the Goldman flows, \cite{Go}, on this space are given by a densely defined 
Hamiltonian system, which corresponds to a torus action on an open dense subset of the moduli space. The moduli 
space is then ``almost'' a toric variety, and the count of integer lattice points in the moment polytope (i.e., 
what would give the dimension of the space of sections on the moduli space, if it were genuinely toric) turns out 
to be given by the Verlinde formulae. The result is formulated in terms of Bohr-Sommerfeld quantization. The 
process fails for higher rank: there are simply not enough Goldman flows.

Recent works of Kaveh and Khovanski in \cite{KK} in the complex set-up and of Harada and 
Kaveh in \cite{HaK} in the symplectic set-up show that such a property of being ``almost 
toric'' is rather prevalent compared to what one might expect. The argument of \cite{HaK} 
proceeds by degeneration into something that is actually toric; in the case considered
by Jeffrey-Weitsman, the paper \cite{HuJ} in fact constructs a candidate toric variety, associated to a 
degeneration of the curve to a nodal curve that consists of the glueing of three-punctured 
spheres (also called ``trinions'').

It is natural to address the question whether such a picture holds in the more general case of moduli space of 
vector bundles of arbitrary rank. This would have the potential of turning into a quite direct verification of 
the Verlinde formulae. The Jeffrey-Weitsman result suggests that the moduli space should be associated to a 
pants decomposition of the surface, or more properly to its holomorphic analogue, namely a degeneration into a 
nodal curve consisting of glueings of three-pointed spheres --- a sort of balloon animal. Thus a first step for us 
will be to consider such a degeneration, of course beginning with the appearance of just one node.

We do not want just any degeneration. The moduli space of vector bundles has a symplectic interpretation via the 
Narasimhan-Seshadri theorem as a space of representations of the fundamental group into ${\rm U}(n)$ \cite{NS}; the 
limit spaces that we want will be toric, and so also have a both holomorphic and symplectic representation. We 
will want our degenerations to proceed in parallel, both in the symplectic and holomorphic category. Indeed, while there are several degenerations that have been studied in the past \cite{Bh2,Bh3,Gi,NagS1,NagS2,Te,Xi}, we want the degeneration to track the symplectic geometry quite closely.

We will also want in the limit that our space on the nodal curve be obtained as the result of a symplectic or 
holomorphic quotient of a moduli space defined on the punctured curve obtained by unglueing the node. Thus, 
proceeding iteratively over the appearance of nodes into a final decomposition of the curve into a glueing of 
trinions (three-punctured spheres), our degenerate moduli space should appear as the result of a symplectic or holomorphic quotient of a product 
of moduli spaces associated to trinions.

We begin in Section \ref{se2} with the symplectic category of representations of the fundamental group
of the Riemann surface. 
Representing the degenerations of the curve on an open set as a family of quadrics in the plane, the degeneration 
is easy to obtain, in terms of flat connections, as the restriction of a singular flat connection in the plane. 
This gives us by the way another desired property, that isomonodromic deformations should give us a holomorphic 
family, well-behaved in the limit.

As a limiting flat connection, we will obtain on the twice punctured surface constructed by 
unglueing the curve a connection with regular singular points at the punctures, with 
opposite residues. The reduction will be a (partial) identification of the fibers at the 
singular points. Such a process has already been discussed earlier in \cite{HuJS, HuJS2}.

We then turn in Section \ref{se3} to give a holomorphic interpretation of all this. Again, referring to \cite{HuJS, 
HuJS2}, we will be deforming the holomorphic bundles into bundles over the desingularised nodal curve equipped 
with a ``framed parabolic structure''. One product of this by-discussion is an emphasis on a particularly 
apposite interpretation of parabolic weights, as the decay rates of sections of a connection with a regular 
singular point; alternately, as the eigenvalues of the residue of the connection at a singular point.

We then turn to moduli; this is a fibered problem, over the family $X_t$ of curves degenerating to a nodal curve 
$X_0$. From the symplectic point of view, this is quite straightforward; one is simply dealing with a family of 
representations of the fundamental group, with some extra structure at $t\,=\,0$. This already gives the family
of moduli spaces as a topological space; we then can concentrate on the holomorphic structure. Fiberwise, this is given by 
the natural Narasimhan-Seshadri correspondence, with a variant introduced in \cite{HuJS2} for $t\,=\,0$. We even 
have a family of holomorphic sections, given by isomonodromic deformation, this is basically enough, even if the 
isomonodromic deformation does not map holomorphic families to holomorphic families.

The paper closes with a discussion of multiple degenerations, allowing occurrence of several nodes.

\section{Symplectic degeneration}\label{se2}

\subsection{Local models for the curves}

Let $U\,:=\, \{z\, \in\, {\mathbb C}\, \mid\, |z|\, <\,1\}$ be the unit disk. 
We consider a holomorphic family
\begin{equation}\label{fx}
p\, :\, \X\, \longrightarrow\, U
\end{equation}
of complex projective curves. We suppose that the fibers $X_t\, :=\, p^{-1}(t)$, for $t\,\neq\, 0$ are smooth 
non-singular projective curves of genus $g$, and that the fiber $X_0\, :=\, p^{-1}(0)$ has a single node. We also suppose 
that $X_0$ is irreducible, so that its desingularisation ${\widetilde X}_0$ is connected. This
assumption implies that the genus of ${\widetilde X}_0$ is actually
$g-1$. It should be clarified that this assumption is made mostly for simplicity of notation;
the general case of disconnected ${\widetilde X}_0$ is essentially the same.

Let $x_1,\, x_2$ be the two points in the desingularisation ${\widetilde X}_0$ that map to the node of $X_0$.

Let $B$ be the polydisk in $\C^2$ given by the product of two disks of radius 2 centered at the origin. We will
assume that a local model for the degeneration to the node is given by the family $\Q\Q$ of quadrics
\begin{equation}\label{eqt}
\Q\Q\, \supset\,
Q_t\,=\, \{(x,\,y)\,\in\, B\,\mid\, xy \,=\, t,\ \ t\,\in\, U\}
\end{equation}
in $B$, so that
there is a neighbourhood $N$ of the node in $\X$ such that $X_t\cap N\,=\, Q_t$. 

The $Q_t$ are cylinders, for $t\,\neq\, 0$; in fact they leave $B$ on curves $$(x(t,\,\theta), \,y(t,\, \theta))
\,=\,\left(2\exp ({\sqrt{-1}\theta}),\, \frac{t}{2}\exp({-\sqrt{-1}\theta})\right)$$ and $$(x(t,\,\theta), \,y(t,\, \theta))
\,=\, \left(\frac{t}{2}\exp({\sqrt{-1}\theta}),\, 2\exp({-\sqrt{-1}\theta})\right)\, .$$
At $t\,=\,0$, the cylinder becomes a
pair of disks, leaving the polydisk $B$ at the circles $$(x(t,\,\theta), \,y(t,\, \theta))
\,=\,(2\exp({\sqrt{-1}\theta}),\,0)\ \text{ and }\ (x(t,\,\theta), \,y(t,\, \theta))
\,=\,(0,\,2\exp({-\sqrt{-1}\theta}))\, .$$
There is a
cycle $\gamma_t$ in $X_t$, given by $$(x(t,\,\theta), \,y(t,\, \theta))
\,=\,\sqrt{t}(\exp({\sqrt{-1}\theta}),\, \exp({-\sqrt{-1}\theta}))\, ,$$ that is vanishing as $t$
approaches $0$. It may be noted that $\gamma_t$
projects to $$x(t) \,=\, \sqrt{t}\exp({\sqrt{-1}\theta})$$ in the $x$ direction, and to $y(t)
\,=\, \sqrt{t}\exp({-\sqrt{-1}\theta})$ in the $y$ direction.

This local model can be glued to the boundaries of the disjoint
union $(U\times S^1)\sqcup (U\times S^1)$
of a family of smooth curves with 
two punctures over $U$ to construct a family $\X$ of curves over $U$, with central fiber $X_0$; suppose that the genus 
of $X_t$, $\ t\,\neq\, 0$, is $g$. We will treat first the case of the two-punctured curves being connected, in which case 
the genus of the blowup ${\widetilde X}_0$ of the central fiber $X_0$ is $g-1$.

Blowing up the origin of our local model gives a holomorphic mapping ${\widetilde B}\,\longrightarrow
\,B$, with a divisor $D$ of self-intersection $-1$ over zero, and similarly a holomorphic
mapping $\pi\,:\, {\widetilde \X}\,\longrightarrow \, \X$; the fiber over zero 
will be ${\widetilde X}_0\cup 2D$ (counting multiplicity), and $X_t$ is the fiber over $t$.

\subsection{Local models for connections}\label{se2.2}

We have, likewise, a model for an isomonodromic family of connections. Consider on the curve 
$Q_{1/2}$ in \eqref{eqt}, say, a local system on a vector bundle of rank $n$ with monodromy $A$ around 
the curve $\gamma_ {1/2}$. We can extend this isomonodromically to the curves $Q_t\,$, 
$\,t\,\in\, U^*$; note that this has a non-trivial monodromy as $t$ moves around the origin, 
but as this monodromy is itself given by a power of $A$, the isomonodromic deformation is well 
defined. The isomonodromic deformation extends to the curve $Q_0$, as long as it is punctured 
at the node.

As, for example, in the work of Malgrange \cite{Mal}, isomonodromic deformations can be represented as flat 
connections on the entire family of quadrics $\Q\Q$, and indeed we can do this here, on the complement of the
coordinate axes. Let $ A \,=\, \exp(-2\pi \sqrt{-1}\vec\alpha)$, where
\begin{equation}\vec\alpha \,=\, diag(\alpha_1, \cdots ,\alpha_n)\end{equation}
is the diagonal $n\times n$ real matrix with diagonal entries
$(\alpha_1,\, \alpha_2,\,\cdots , \,\alpha_n)$ (the $(i,\, i)$-th entry
is $\alpha_i$) belonging to the simplex 
$$ 
\Delta = \{(\alpha_1,\, \alpha_2,\,\cdots , \,\alpha_n)\,\mid\,\alpha_1\,\geq\,\alpha_2\,\geq\,\cdots \,
\geq\,\alpha_n\,\geq \, \alpha_1 -1,\ \sum_{i=1}^n
\alpha_i\,=\, 0\, \}\, .$$
Over the polydisk $B$ in \eqref{eqt}, we consider the flat connection on a unitary bundle on $x\,\neq\, 0,\
y\,\neq\, 0$
$$\nabla \,=\, d + \sqrt{-1} \frac{\vec\alpha}{2} (d\theta_x -d\theta_y)\, =\,\partial +
\frac{\vec\alpha}{4}\left(\frac{dx}{x} -\frac{dy}{y}\right)+ \overline \partial +\frac{\vec\alpha}{4}
\left(\frac{-d\overline x}{\overline x} +\frac{-d\overline y}{\overline y}\right)\, ,$$
where $\theta_x$ (respectively, $\theta_y$) is the angular coordinate in the
$x$ (respectively, $y$) variable. Regarding the monodromy of the connection $\nabla$, the following hold:
\begin{itemize}
\item the monodromy around $x\,=\,0$ is $A^{1/2}\,=
\,\exp( -\pi\sqrt{-1}\vec\alpha)$,

\item the monodromy around $y\,=\, 0$ is $A^{-1/2}$, and

\item integrating
along $x\,=\, r\exp({\sqrt{-1}\theta}),\ y \,=\, r^{-1}t \exp({-\sqrt{-1}\theta})$,\
$\theta\,\in\, [0,\, 2\pi]$, that is on the cycle $\gamma_t$ on the curve $Q_t$, the
monodromy is $A$.
\end{itemize}
In coordinates $(x,\, t)$, the connection $\nabla$ becomes
$$\nabla \,=\, d + \sqrt{-1} \frac{\vec\alpha}{2}\left(2 d\theta_x -d\theta_t\right)$$
and in coordinates $(y,\,t)$, it is
$$\nabla \,=\, d +\sqrt{-1} \frac{\vec\alpha}{2}\left(-2 d\theta_y +d\theta_t\right)\, ,$$
and projecting out the normal ($d\theta_t$) components, we have a partial connection $\nabla^p$
over the curves. On the patch $x\,\not=\, 0$, taking $x$ as the coordinate: $$\nabla^p \,=\, d +
\sqrt{-1}\vec \alpha d\theta_x\, .$$ The partial connection extends well
to the limit $y\,=\,0$, despite the singularity of $\nabla$. Similarly, on $y\,\not=\, 0$, the
partial connection becomes $$\nabla^p \,=\, d -
\sqrt{-1}\vec\alpha d\theta_y\, ,$$ which again passes to the limit.
 
For use in the sections to follow, we consider connections compatible with the holomorphic
structure. Change gauge in the vector bundle by 
$$G\,=\, r_x^{-\vec\alpha/2} r_y^{\vec\alpha/2}\, ,$$ where $r_x^{\vec\alpha/2}$
(respectively, $r_y^{-\vec\alpha/2}$) is constructed using the matrix $\vec\alpha$ and the radial
function $r$ in $x$ (respectively, $y$) coordinate. This changes our connection by adding to it the term
$$-(dG) G^{-1} \,=\, \frac{\vec\alpha}{2}\frac{dr_x}{r_x}-\frac{\vec\alpha}{2}\frac{dr_y}{r_y}\, .$$
Now note that $\frac{dr }{r }\,=\, \frac{1}{2}(\frac{dz }{z }+\frac{d\overline z }{\overline z })$
and $d\theta \,=\,\frac{1}{2\sqrt{-1}}(\frac{dz }{z } - \frac{d\overline z }{\overline z })$, so that
the resulting connection is 
$$\nabla\,=\, d+\frac{\vec\alpha}{2}\left(\frac{dx }{x } -\frac{dy}{y}\right)
\,=\, d + \frac{\vec\alpha}{2}\left(\frac{2dx }{x } -\frac{dt }{t }\right)\,=\, d +
\frac{\vec\alpha}{2}\left(\frac{-2dy }{y} +\frac{dt }{t }\right)\, .$$
This gives a holomorphic gauge (so no $(0,\,1)$ type component), but one in which the bases are 
decaying or blowing up like $r_x^{\frac{\vec\alpha}{2}}r_y^{-\frac{\vec\alpha}{2}}$; we note 
that along the curves, the decay rate is $r_x^{\vec\alpha}$ along $x\,\neq\, 0$, and 
it is $r_y^{-\vec\alpha}$ along $y\,\neq\, 0$. Projecting out the $dt$--component, we have a holomorphic 
family of partial connections (i.e., holomorphic connection just along the curves), which pass 
well to the limit $t\,=\,0$, where there is a connection with poles on both patches at $x\,=\,0,\ 
y\,=\,0$; the residues are opposite of each other, modulo an integer.

The above local model can be used in deforming flat connections more globally. For example, choose 
a polystable holomorphic ${\rm SL}(n, \C$) vector bundle $E_{1/2}$ of degree zero, on the curve, say, $X_{1/2}$. By the 
Narasimhan-Seshadri theorem \cite{NS}, this corresponds to a flat unitary connection on $X_{1/2}$; this connection is denoted
by $\nabla_{1/2}$. We can then deform this connection isomonodromically to connections $\nabla_t$ on all of the 
$X_t$, $\,t\,\neq\, 0$, and also to $X_0$ away from the nodal point $(x,\,y) \,=\, (0,\,0)$. Indeed, over
$\X \setminus B$,
this is straightforward. We can then glue this to the partial connection given by our local model. Note that if
$A$ is the holonomy of the connection along the cycle $\gamma_t$, then on the limiting curve $X_0$ we 
obtain a connection with holonomy $A$ around the origin on the curve $y\,=\, 0$, and holonomy $A^{-1}$ 
around the origin on the curve $x\,=\,0$. This tells us that in the limit $t\,=\,0$, we have, along the 
two branches of the curve meeting at the node, opposite holonomies at the two punctures 
corresponding to the node; on the desingularisation, there are two parabolic points, with 
holonomy in opposite (inverse) conjugacy classes.

It is useful to consider the lift of this connection to the blow-up $\widetilde B$ of $B$ at the origin. This amounts to setting (on one coordinate patch) 
\begin{equation}\label{blowup1}
x \,=\, \widetilde x,\ \ y \,= \,\widetilde x \widetilde y\,;
\end{equation} 
the connection becomes, in terms of the holomorphic trivialization,
$$\nabla \,=\, d -\frac{ \vec\alpha}{2}\frac{d\widetilde y }{\widetilde y }\, , $$
and on the ``opposite'' patch, putting
\begin{equation}\label{blowup2}
x \,= \,\widehat x \widehat y,\ \ y \,=\,\widehat y\, ,\end{equation}
the connection becomes
$$\nabla \,=\, d + \frac{ \vec\alpha}{2}\frac{d\widehat x }{\widehat x }\, . $$
We therefore have poles along two disjoint divisors, with opposite polar parts.

\subsection{Spaces of representations}

We will restrict our attention to the group ${\rm SU}(n)$. Let $T$ be the standard diagonal 
torus in ${\rm SU}(n)$. Section \ref{se2.2} tells us that we should have, by isomonodromic 
deformation, a way of degenerating the character variety of ${\rm SU}(n)$ representations 
of the fundamental group of $X_t,\ t\,\neq\, 0$, to a space of representations of the 
fundamental group of the desingularised curve $\widetilde X_0$, with the constraint of 
opposite holonomy at the two punctures. Note that if we count parameters for ${\rm SU}(n)$ 
representations with the monodromy $A$ generic, we have for $X_t,\ t\,\neq\, 0$, exactly 
$(2g-2)(n^2-1)$ parameters for the representation space. Indeed, we have $2g(n^2-1) $ 
parameters corresponding to the generators of the fundamental group, minus $n^2-1$ 
parameters of constraints for the relation of the fundamental group of the curve, minus 
$n^2-1$ for quotienting by conjugation, giving the number $(2g-2)(n^2-1)$. For the curve $X_0$, we 
would get the following:
\begin{itemize}
\item $2g(n^2-1)$ parameters for the generators of the fundamental group of the punctured 
desingularised curve $\widetilde X_0$,

\item minus $n^2-1$ parameters of constraints from the relation of the generators of the 
fundamental group of $\widetilde X_0$,

\item minus $n^2-1$ for quotienting by conjugation,

\item minus a further $n-1$ constraints given by the constraint that conjugacy classes of 
holonomy should be opposite at the two punctures,

\item plus an additional $n-1$ parameters, as we must glue back together the desingularised 
curve to the nodal curve, and then the bundle over it so that the holonomies on the two branches 
of the node are of the form $A,\ A^{-1}$ with $A\,\in\,{\rm SU}(n)$. Generically, this glueing is only 
determined by the holonomy up to an element of the stabilizer of $A$, and so there are $n-1$ 
parameters (the dimension of the stabilizer).
\end{itemize}

In short this gives again, but in a different way, $(2g-2)(n^2-1)$ parameters for the 
representations over $X_0$. Note that on the desingularisation $\widetilde X_0$, fixing the 
conjugacy class of $A$, we have the symplectic version of the moduli space of 
parabolic bundles. The space we are seeking is the union of these parabolic moduli spaces,
together with some framing 
parameters which generically lie in the Cartesian product $(S^1)^{n-1}$. The framing parameters
are to be thought of as symplectically dual to the conjugacy 
class parameters in $\Delta$.

This na\"{\i}ve picture works well, as long as the conjugacy class of $A$ is generic, that is,
the eigenvalues are distinct. 
When $A$ is non-generic, things go awry, and we no longer have a symplectic space. There is a 
solution to this problem, given by quasi-Hamiltonian implosion, as developed in \cite{HuJS}.

To explain the above mentioned quasi-Hamiltonian implosion, first consider the space of flat ${\rm SU}(n)$ connections
$Conn$ on the punctured desingularised curve ${\widetilde X}_0$, equipped with framings at the two punctures: 
\begin{align}
Conn \,= \{(A_1,\, A_2,&\,B_1,\, B_2,\, (C_i,\, D_i)_{ i= 1,\cdots ,g-1})
\,\in \,{\rm SU}(n)^{2g+2}\,\nonumber\\ &
\big\vert\ (\prod_i^{g-1}[C_i,\, D_i])B_1A_1 B_1^{-1}B_2A_2B_2^{-1}
\,=\, 1\}/{\rm SU}(n)\, .\label{eqcon}\end{align}
Here $A_j$ represents the holonomy at the $j$-th puncture in the trivialization given there, and
$B_j$ represents the parallel transport from the puncture, in the framing given there, to a base point, while 
$C_i,\,D_i$ represent transport around a standard homotopy basis, starting from a
base point. We then quotient by the natural action derived from a change of gauge at the base point.

The implosion
\begin{equation}\label{EC}
C\,=\, Conn_{impl}
\end{equation}
is given by 
\begin{itemize} 
\item restricting the matrices $A_1,\, A_2$ to make them lie in the fundamental alcove 
$\mathcal A$ of the diagonal torus (which represents the set of all conjugacy classes), so that $A_1,\, A_2$ are diagonal, and 
then

\item imposing an 
equivalence relation on these spaces, given by identifying to points the orbits of the elements $B_1,\, B_2$ under the 
right action of the commutators
$$[{\rm Stab}(A_1),\ {\rm Stab}(A_1)]\, ,\ \ [{\rm Stab}(A_2),\ {\rm Stab}(A_2)]$$ of the
stabilizers of $A_1,\, A_2$ under conjugation. 
\end{itemize}
Note that 
for $A_j$ generic, this commutator is trivial. More generally, if $${\rm Stab}(A_j)\,=\,
S({\rm U}(n_1)\times\cdots\times{\rm U}(n_k))\, ,$$
then the commutator is 
$${\rm SU}(n_1)\times\cdots\times{\rm SU}(n_k).$$
 From
these it follows that instead of having, in addition to the 
representation of the fundamental group, a unitary framing of each eigenspace of the $A_i$, the implosion process 
actually leaves us instead with the lesser information of a framing of the top exterior power of each generalized
eigenspace. However, for 
generic (meaning distinct) eigenvalues, this is the same information. There is also a holomorphic interpretation of 
the space $C$ (defined in \eqref{EC}), given in \cite{HuJS2}, to which we shall return later.

On $C$ in \eqref{EC}, there is a Hamiltonian action of two copies $T\times T$ of the diagonal matrices, given by 
$$(t_1, \,t_2) (A_1,\, A_2,\,B_1,\, B_2,\, (C_i,\, D_i)_{i= 1,\cdots ,g-1})
\,= \,(A_1,\, A_2, \,B_1t_1,\, B_2t_2,\, (C_i, \,D_i)_{i= 1,\cdots ,g-1}).$$
This action, in essence, is on a framing at the two punctures. The action is symplectic, with moment map for the
action being simply $(A_1, \,A_2)$, in the standard linear parametrization of the fundamental alcove,
realized as the exponential of an isomorphic alcove in the Lie algebra.
Summarizing we have:

\begin{theorem}[{\cite{HuJS}}]
The space $C$ in \eqref{EC} is a ``master moduli space'' for the representations of the fundamental group of
${\widetilde X}_0$ into ${\rm SU}(n)$. It
is a symplectic manifold with singularities, and is equipped with a Hamiltonian action of $T\times 
T$ with a decomposition into symplectic strata, labelled by the cells of the fundamental alcove (essentially, the 
multiplicity patterns of the eigenvalues of $A_j$). The symplectic reductions at $$(A_1 ,\, A_2 )\,=\, (\exp(2\pi
\sqrt{-1}\vec\alpha_1) ,\,\, \exp(2\pi\sqrt{-1} \vec \alpha_2)) \,\,\in\,\, ({\mathcal A}_1,\,{\mathcal A}_2)$$
are the moduli spaces of representations of the fundamental group of
$\widetilde X_0$, with fixed conjugacy classes $(A_1 ,\, A_2 )$ at the 
punctures.\end{theorem}

We take instead a symplectic quotient of $C$ in \eqref{EC} by the anti-diagonal torus $T_{ad}$ in $T\times T$; this 
sets, in the alcove parametrization, $A_2 \,=\, R(A_1^{-1})$, where $R$ reverses the order of the diagonal entries, 
and identifies $(B_1,\, B_2)$ with $(B_1\tau,\, B_2\tau^{-1})$ for all $\tau\,\in\, T$.

We now recall a result of \cite{HuJ}.

\begin{proposition}[{\cite{HuJ}}]
The symplectic quotient $C/\!\!/T_{ad}$ of the space of ``imploded'' representations $C$ of the fundamental group of
the punctured desingularised curve $X_0$ is given by
\begin{align}
M_0 \,=\,\{(A ,\,B_1,\, & B_2,\, (C_i,\, D_i)_{i= 1,\cdots ,g-1})\,\in\, \mathcal{A}\times
{\rm SU}(n)^{2g}\nonumber\\
&\big\vert\ (\prod_{i=1}^{g-1}[C_i,D_i])B_1A B_1^{-1}B_2A^{-1}B_2^{-1}\,=\, 1\}/\simeq\, ,
\label{ae}
\end{align}
where 
$$(A,\,B_1,\, B_2,\,(C_i,\, D_i))\,\simeq\,(A,\, {\bf g}B_1\tau \mu,\, {\bf g}B_2\tau\nu,\, ({\bf g}C_i{\bf g}^{-1},\,
{\bf g}D_i{\bf g}^{-1}))$$
for ${\bf g}\,\in\, {\rm SU}(n), \, \tau\,\in\, T$ and $\mu, \, \nu \,\in\,
[{\rm Stab}(A),\, {\rm Stab}(A)]$. 
\end{proposition}

The matrix $B^{-1}_1B_2$ in \eqref{ae} represents parallel transport along the glued curve, from one side of the node to the 
other. Under the action of $t\,\in\, T$ it is mapped to $t^{-1}B^{-1}_1B_2t$.
Since $B_i$ represents parallel transport from the base point to the $i$--th puncture, it
follows that $B^{-1}_1B_2$
represents parallel transport from the first puncture to the second one. Once these two punctures have been glued,
so that the glued point corresponds to a pinched cycle, $B^{-1}_1B_2$ is the transport from one side of the
pinched cycle to the other side, the long way round.

For generic $A$ (meaning distinct eigenvalues for it), the equivalence in the quotient basically 
matches the corresponding eigenspaces at the two punctures, and gives us a well defined 
bundle on $X_0$, with the fiber over the singular point decomposing into a direct sum of 
eigenspaces, with eigenvalues on each branch of the curve that are inverses of each other. 
When the eigenvalues are not distinct, the identification is only one of the top exterior 
powers of eigenspaces, and so there is no well defined bundle, rather only a bundle on both branches of 
the curve at the singular point, with only a partial identification, of the top exterior 
powers of the respective eigenspaces.

Returning to the family of curves in \eqref{fx} over the disk $U$, we have, symplectically, a picture of
the family $\M$ of moduli spaces for which we would like the following: the space $M_t$ is the
space of ${\rm SU}(n)$ representations of the fundamental group of $X_t$
over $t\,\neq\,0$, and over $t\,=\,0$, the space $M_0$ is the space of ``imploded'' representations of the fundamental group of the punctured
desingularised curve with ``opposite" holonomy at the two punctures, with some extra glueing parameters associated
to the punctures.

We note that $X_0$ is the curve obtained by pinching to zero the
curves $\gamma_t$, and that these curves can be realized as part of a symplectic basis $c_i,\, d_i$ for the
fundamental group, say as the path $c_g$. In the limit, $c_g$ corresponds to the path around the puncture
along $y\,=\,0$ in $X_0$, and $c_g^{-1}$ to the path around $x\,=\,0$, with $d_g$ being the path linking the
two punctures on the desingularisation. To obtain a uniform construction of $M_t$, we describe $M_t$ for
$t\,\neq\, 0$ by using trivializations at two points, instead of one, and describe the representation
as usual by parallel transport around the cycles or between the base points. Let the first point $p_1$
be the standard base point, and let the second $p_2$ be the base point near the eventual puncture. With this, the
matrix $A$ will correspond to the cycle $c_g$ now originating at $p_2$, the matrix $B_1$ to transport from
$p_1$ to $p_2$, and the matrix $B_2$ to the path along $d_g$, starting at $p_1$ and ending at $p_2$. Then
the representation space is 
\begin{align}M_t \,=\,\{(A ,&\,B_1, \, B_2, \,(C_i,\, D_i)_{ i= 1,\cdots ,g-1} )\in\, \mathcal{A}\times
{\rm SU}(n)^{2g}\nonumber\\ &\big\vert \,\,\, (\Pi_i[C_i,D_i])B_1A B_1^{-1}B_2A^{-1}B_2^{-1}
\,=\, 1\}/\simeq\, .\nonumber\end{align}
with now the equivalence relation being 
$$(A ,\,B_1,\,B_2,\,(C_i,\, D_i))\,\simeq\,(A,\, {\bf g}B_1\mu,\, {\bf g}B_2\mu,\, ({\bf g}C_i{\bf g}^{-1},\,
{\bf g}D_i{\bf g}^{-1}))$$
for ${\bf g}\,\in\, {\rm SU}(n)$ and $\mu \,\in\, {\rm Stab}(A)$.
For $A$ generic this is the same equivalence relation as for $M_0$, since ${\rm Stab}(A)\,=\,
T$ and $[{\rm Stab}(A),\, {\rm Stab}(A)]\,=\,1$; in general, one sees that at $M_0$, instead of quotienting
by ${\rm Stab}(A)$, one is quotienting by the bigger group 
$$T\times [{\rm Stab}(A),\, {\rm Stab}(A)]\times [{\rm Stab}(A), \,{\rm Stab}(A)]\, .$$

Now we can build the space $\M$ in a uniform way as follows: set
\begin{align}
\M \,=\, \{(A ,\,B_1&, B_2,\, (C_i,\, D_i)_{ i= 1,\cdots ,g-1},\, t)\,\in\, \mathcal{A}\times
{\rm SU}(n)^{2g}\times U\nonumber\\
&\,\big\vert\,\,\, (\Pi_i[C_i,D_i])B_1A B_1^{-1}B_2A^{-1}B_2^{-1}
\,=\, 1\}/\simeq\label{MM-sympl}
\end{align}
with the equivalence relation being 
$$(A ,\,B_1,\,B_2,\, (C_i,\, D_i), \,t)\,\simeq\,(A,\, {\bf g}B_1\mu,\,{\bf g}B_2\nu,\, ({\bf g}C_i{\bf g}^{-1},\,
{\bf g}D_i{\bf g}^{-1}),\,t)$$
where
\begin{itemize}
\item ${\bf g}\,\in\, {\rm SU}(n)$,

\item $\mu \,=\,\nu \,\in\, {\rm Stab}(A)$ at $t\,\neq\, 0$, and 

\item at $t\,=\, 0$, $$\mu \,= \,\tau\mu'\,\ \ \text{ and }\,\ \ \nu \,= \,\tau\nu',$$ where 
$\tau\,\in\, T$ and $\mu',\, \nu' \,\in\, [{\rm Stab}(A),\, {\rm Stab}(A)]$.
\end{itemize}

We can also consider the case where the desingularisation $\widetilde X_0$ of $X_0$ is disconnected, giving two 
components $X_{0,1},\, X_{0,2}$ of genus $h$ and $g-h$, each with one puncture. Again, we have framed connection spaces
\begin{align}\mathcal C_1 &\,=\,\{(A ,\, B ,\,(C_i,\, D_i)_{ i= 1,\cdots ,h} )\,\in\,
{\rm SU}(n)^{2h+2}\,\,\big\vert\,\,
(\Pi_i[C_i,\,D_i])B A B ^{-1}\,=\, 1\}/{\rm SU}(n),\nonumber\\
\mathcal C_2 &\,=\,\{(A ,\,B ,\,(C_i, \,D_i)_{ i= 1,\cdots ,g-h} )\,\in\, {\rm SU}(n)^{2(g-h)+2}
\,\,\big\vert\,\, (\Pi_i[C_i,\,D_i])B A B^{-1}\,=\, 1\}/{\rm SU}(n).\nonumber\end{align}

We can take the implosion associated to the puncture for each of these spaces; they will have
Hamiltonian $T$ actions associated to each of their punctures, and so one can reduce, as above, by
the diagonal $T$. We get 
\begin{align}
M_0 \,=\,\bigg\{(A ,\,B_1,\,B_2,\, &(C_i,\, D_i)_{ i= 1,\cdots ,g} )\,\in\, \mathcal{A}\times
{\rm SU}(n)^{2g}\, \mid\, (\Pi_{i=1}^h[C_i,D_i])B_1A B_1^{-1}
\nonumber\\
&= \,1 \,=\, (\Pi_{i=h+1}^g[C_i,\,D_i])B_2A^{-1}B_2^{-1}\bigg \}/\simeq
\nonumber\end{align}
where 
$$
(A ,\,B_1,\, B_2,\,(C_i,\, D_i)_{i=1,\cdots ,h},\,(C_i,\, D_i)_{i=h+1,\cdots ,g})\,\simeq
$$
$$
(A,\, g_1B_1t \mu, \,g_2B_2t\nu,\, (g_1C_ig_1^{-1},\, g_1D_ig_1^{-1})_{i=1,\cdots ,h},\,
(g_2C_ig_2^{-1},\,g_2D_ig_2^{-1})_{i=h+1,\cdots ,g}))
$$
for $g_1,\, g_2\,\in\, {\rm SU}(n),\, \tau\,\in \,T$ and $\mu, \,\nu \,\in\, [{\rm Stab}(A),\, {\rm Stab}(A)]$.

Again, it is possible to give a construction of $M_t$, $t\,\neq\, 0$, and of $\M$ along the same lines.

\section{Holomorphic models for parabolic moduli}\label{se3}

\subsection{Parabolic bundles and fixed weights on $\widetilde X_0$}

We would like to have a holomorphic model for these spaces. On the curves $X_t$, $t\,\neq\, 0$, we are simply 
dealing with the moduli space of polystable vector bundles, equipped with an ${\rm SL}(n,\C)$ structure, that is a holomorphic volume form. We turn to considering our bundles and connections over the curve 
$X_0$, and also over its desingularisation $\widetilde X_0$.
 
Our holonomies at the punctures live in the fundamental alcove for the group. Writing this alcove $\mathcal A$ as the set of 
$$A \,=\,\exp(-2\pi\sqrt{-1}
\vec\alpha) \,=\,\exp(-2\pi\sqrt{-1} diag(\alpha_1,\, \alpha_2,\,\cdots , \,\alpha_n))\, ,
$$
where $(\alpha_1,\, \alpha_2,\,\cdots , \,\alpha_n)\,\in\, {\mathbb R}^n$ belongs to the simplex 
$$ 
\Delta = \{(\alpha_1,\, \alpha_2,\,\cdots , \,\alpha_n))\,\,\big\vert\,\,\alpha_1\,\geq\,\alpha_2\,\geq\,\cdots \,\geq\,\alpha_n\,\geq \, \alpha_1 -1,\, \sum_{i=1}^n
\alpha_i\,=\, 0\, \},$$
giving $\mathcal A \,\simeq \,\Delta$. We have that the logs of the holonomies correspond in the holomorphic category to
the {\it weights} for a parabolic structure. The simplex $\Delta$ divides into faces $\Delta^{(I, k)}$ indexed by
$$(I,\,k) \,=\, ((I_1,\, I_2,\,\cdots,\, I_\ell),\, k),\ 0\,<\,I_1\,<\, I_2\,<\,\cdots\,<\,I_\ell\,=\,n,\
k\,=\,0,\,1\, ,$$ where $k\,=\,0$ if $\alpha_1\,<\,\alpha_n+1$,\, $\,k\,=\,1$ if $\alpha_1 \,=\, \alpha_n+1$, and
the $I_j-I_{j-1}$ are the multiplicities of the $\alpha_i$ so that we have the pattern of multiplicities 
$$\alpha_1\,=\,\ldots\,=\,\alpha_{I_1}\,>\, \alpha_{I_1+1}\,=\,\ldots\,= \, \alpha_{I_2}
\,>\, \alpha_{I_2+1}\,=\,\cdots\,=\,\alpha_{I_{\ell -1}}
$$
$$
> \,\alpha_{I_{\ell-1} +1}\,=\,\ldots\, =\,\alpha_{I_\ell}
\,=\,\alpha_n\, .$$
Note that away from $\alpha_n \,=\, \alpha_1-1$, the eigenvalue multiplicities in $\mathcal A$ and $\Delta$ are the
same; but on $\alpha_n\,=\, \alpha_1-1$, different eigenvalues for $\vec\alpha$ can give the same eigenvalue for $A$.

Correspondingly, we decompose $C$ in \eqref{eqcon} as a union of strata
\begin{equation}\label{l1}
\bigsqcup_{(I^1, k^1), (I^2, k^2)} Conn_{(I^1, k^1), (I^2, k^2)},
\end{equation}
obtained by fixing the eigenvalue multiplicities of the holonomies $A_i$
around the punctures $x_i$. As we are interested in the case where $A_1\,=\,A_2^{-1}$, and so
$\vec\alpha^1\,=\, -\vec\alpha^2$, or more properly 
$$\vec\alpha^1\,=\, -R(\vec\alpha^2)\, ,$$
where $R$ is the reversal symmetry, we will be interested in strata $\Delta^{(I_1, k_1)}\times
\Delta^{(I_2, k_2)}$ with $k_1\,=\, k_2$, and $I^1\,=\, \widehat R(I^2)$; here $\widehat R$ is the
involution induced by $R$. Set 
$$
\Delta^R \,=\, \bigsqcup_{(I , k )}\Delta^{(I , k )}\times\Delta^{(\widehat R(I), k)}\, .
$$
Note that the generic stratum is included in $\Delta^R$. In a similar vein, consider over
$\Delta^{(I_1, k_1)}\times\Delta^{(I_2, k_2)}$ the stratum 
$Conn_{(I^1, k^1), (I^2, k^2)}$ (see \eqref{l1}) with $k^1\,=\, k^2$ and $I^1\,=\, \widehat R(I^2)$. We will thus use just the
index $(I^1,\, k^1)$ associated to $x_1$ and refer to these strata as $Conn_{(I^1, k^1)}$. We set 
$$
Conn^R \,=\, \bigsqcup_{(I , k )}Conn_{(I, k)}\, .
$$
Again, the generic stratum is included in $Conn^R$.

On the desingularisation $\widetilde X_0$, fixing the $\alpha_i$, for $\alpha_n\,>\, \alpha_1 -1$ (so that $k\,=\,0$), and quotienting out the framings at the
punctures (in essence, this is a symplectic reduction) we have, by a well known theorem of Mehta and Seshadri \cite{MS},
a map from the corresponding space of connections to a holomorphic moduli space of 
parabolic ${\rm SL}(n,\C)$ vector bundles $E$ with parabolic weights
$\alpha_{I_1}\,>\,\alpha_{I_2}\,>\,\cdots \,>\,\alpha_{\ell}$ (with $\alpha_{I_j}$ of multiplicity $I_j-I_{j-1}$, so that $I_0 =0, I_\ell = n$) attached to a flag of subspaces of $E_{x_1} $ of dimensions
$I_1,\, I_2,\,\cdots,\, I_\ell $ at $x_1$, and opposite parabolic
weights $-\alpha_{ \ell} \,>\, -\alpha_{\ell-1} \,>\,\cdots \,>\,-\alpha_{1}$ attached to a flag
at $x_2$ with dimensions $ I_{\ell} -I_{\ell-1},\, \cdots,\, I_{\ell}-I_1,I_\ell$.

For $k\,=\,1$ (so $\alpha_n\,=\, \alpha_{I_\ell}\,= \,\alpha_{I_1}-1 \,=\,\alpha_1-1$), the sheaf, as shown in 
\cite{HuJS2}, if it corresponds to a closed semistable orbit in the GIT quotient (other sheaves will degenerate 
to such closed orbits), actually acquires a torsion component at $x_1$ of length $t^1 \,=\,n-I_{\ell-1}
\,=\,I_\ell-I_{\ell-1}$ and at $x_2$ a torsion component of length 
$t^2\,=\, I_1$, and the ${\rm SL}(n,\C)$ structure degenerates, vanishing at $x_1,\, x_2$ to
order $I_\ell-I_{\ell-1}$,\, $I_1$
respectively. Set $$E'\,=\, E/{\rm torsion};$$ so $E'$ has degree $-I_1 - I_\ell+I_{\ell-1}$. The flags 
shift in consequence: $E'$ has a flag
at $x_1$ of dimension $I_1+t^1, I_2 + t^1, \cdots$, corresponding to shifted weights $\alpha_n+1
\,=\, \alpha_{I_1}\,>\,\alpha_{I_2}\,>\,\cdots \,>\,\alpha_{I_{\ell-1}}$,
and a flag at $x_2$ of dimension $I_\ell-I_{\ell-1}+t^2, \,I_\ell-I_{\ell-2}+ t^2,\,\cdots$, corresponding to
shifted weights $-\alpha_{I_1}+1\,=\, -\alpha_{I_\ell}\,>\,-\alpha_{I_{\ell-1}}\,>\,\cdots \,>\,-\alpha_{I_2}$.
Thus, as $\alpha_1- \alpha_n$ goes to $1$, the degree zero 
${\rm SL}(n, {\mathbb C})$ bundles acquire a fixed torsion piece, and we are left, modulo torsion, with a bundle of
degree $-t^1-t^2$, 
an ${\rm SL}(n,\C)$-structure that degenerates at the support of the torsion, and a shifted flag. The torsion 
piece does not contribute to the modulus, and is basically ``filler''.

The successive quotients of the flags at $x_1$ or $x_2$ correspond to the eigenspaces 
associated to the distinct eigenvalues of $\vec\alpha$ at $x_1$ and of $-\vec\alpha$ at $x_2$, 
with the first quotient corresponding to the largest eigenvalue, and so on.

\subsection{Varying weights on $\widetilde X_0$}

Now let the weights vary. Symplectically, when we wanted to consider all possible $A$, we were faced with the issue of 
varying eigenvalues and varying their multiplicities; in moving to the space $C$ in \eqref{EC}, we incorporated not only
variations in $A$, 
but in addition the dual data of a torus-valued framing. The holomorphic version of this was discussed in the paper 
\cite{HuJS2}; the answer, holomorphically, consists in replacing the flags $E^i_{d_\ell}$ by decomposable elements 
$\beta^i_j$, $\ i\,=\, 1,\,2$, $\ j \,=\, 1,\,\cdots ,\,n$, of $\bigwedge^j(E\big\vert_{x_i})^*$. In essence, this is
the Pl\"ucker correspondence: the correspondence with the flag is given by 
$$E^i_{d_\ell}\,=\, Ann(\beta^i_{n-d_\ell})$$
for $\beta^i_{n-d_\ell} \,\neq\, 0$; the complex scale of these elements, i.e., both radius (encoding weights) and 
arguments (encoding framing), gives us our extra data. Whereas before we had a torus, acting on the framings, we now 
have its complexification, acting on $\beta^i_j$. In the symplectic picture, the norm of $\beta^i_j$ should be 
thought of as the difference $\alpha^i_j - \alpha^i_{j+1}$ of successive weights. In the holomorphic set-up, while 
we do not have the actual values of the weights, the stratification by multiplicity of the eigenvalues (vanishing of 
the $\beta^i_j$) actually does make sense, as it encodes the type of flag.

Indeed, one of the features of the construction is that the $\beta^i_j$ are allowed to vanish as one moves in the 
moduli space, expressing the change in the holomorphic framework from eigenvalues of lower multiplicity to higher 
multiplicity (i.e., merging of eigenvalues). Full flags correspond to all the $\beta^i_j$ being non-zero; if 
$\beta^i_{n-d}\,=\,0$, there is no subspace of dimension $d$ in the flag at $x_i$. Also to have flags (nested 
subspaces) instead of random subspaces, the successive non-zero $\beta^i_j,\ \beta^i_{j'}$ must be {\it compatible}, 
in the sense that there are $\gamma^i_{j,j'}\, \in\, \bigwedge^{j'-j}(E\big\vert_{x_i})^*$ with
\begin{equation}\label{eg}
\beta^i_{j'}\,=\, \gamma^i_{j,j'}\wedge \beta^i_{j}\, .
\end{equation}
The $\gamma^i_{j,j'}$ will be what defines the volume form on $E^i_{n-j}/ E^i_{n-j'}$. Finally, the top term 
$\beta^i_n$ should be compatible with the ${\rm SL}(n,{\mathbb C})$ structure (see \cite{HuJS2}); this says that it 
encodes the top exterior form on the bundle at the point.

We will have a stratified moduli space --- stratified according to the vanishing pattern of the $\beta^i_j$ --- and this 
corresponds to the multiplicity patterns $I_1\,<\, I_2\,<\,\cdots\,<I_\ell$ of the eigenvalues. There is one more
$\beta$ that can vanish and that is $\beta_n$; this will
correspond to $k\,=\,1$, that is when $\alpha_1\,=\,\alpha_n+1$. Recall that we said 
that the bundles could acquire torsion; indeed, on these strata, the closed orbits in the GIT problem are those 
where this torsion is as big as it can get, and with rank given at $x_1$ by the multiplicity of $\alpha_n$, that is 
$I_\ell-I_{\ell-1}\,=\, n-I_{\ell-1}$, and at $x_2$ by the multiplicity of $-\alpha_1$, that is $I_1$. (Note that the 
${\rm SL}(n,\C)$ structure still defines a volume form on the fibers of $E/{\rm torsion}$ at the $x_i$.) Even when 
$k=0$ (no torsion), the multiplicity pattern yields that $$\beta^1_{I_\ell-1},\,\cdots,\, \beta^1_{I_\ell-I_{1}-1}\,=\, 0,\,\ \
\beta^1_{1},\,\cdots,\, \beta^1_{I_\ell-I_{\ell-1}-1}\,=\,0,$$ as well as
$$\beta^2_{n-1},\,\cdots ,\,\beta^2_{n-I_{\ell}
+I_{\ell-1}-1}\,=\, 0,\,\ \ 
\beta^2_{1},\,\cdots,\,\beta^2_{I_1-1}\,=\,0;$$ when $k\,=\,1$, in addition, $\beta^1_n\,=\, \beta^2_n\,=\,0$. This
unites the strings of 
vanishing $\beta$ (cyclically) at the two points, so that there is a run at both points $x_1$ and $x_2$ of vanishing 
$\beta$ corresponding to the fact that the values of $\exp(\alpha_1),\, \exp(\alpha_n)$
(and so of $\exp(-\alpha_1),\,
\exp(-\alpha_n)$) are the same, and giving a common eigenspace of dimension $I_1+ I_\ell-I_{\ell-1}$. The same holds 
at $x_2$.

At $x_1$, as noted above, we have a torsion submodule of dimension $t^1\,=\, n-I_{\ell-1}$, which turns out to be necessarily 
of the form $\C_{x_1}^{t^1}$. Suppose that $s_1,\, s_2,\,\cdots ,\, s_{t^1} $ generate this torsion, then stability 
forces $\beta^1_{t^1}(s_1,\, s_2,\,\cdots,\, s_{t^i}) \,\neq \,0.$ This allows us to define for $j\,>\,t^1$ forms on 
$E/$torsion at $x_1$
$$\widetilde \beta^1_{j-t^1}(v_1,\cdots, v_{j-t^1}) \,=\, \beta^1_{t^1}(s_1,\, s_2,\,\cdots,\, 
s_{t_i}, v_1,\cdots, v_{j-t^1}).$$
Taking annihilators, instead of spaces of dimension $I_1,\,\cdots,\,I_\ell$, as noted, one 
has spaces of $E/$torsion at $x_1$ of dimension $I_1+t^1,\, I_2 + t^1,\,\cdots $, corresponding to the shifted weights 
$\alpha_n+1\,=\, \alpha_{I_1}\,>\,\alpha_{I_2 }\,>\,\cdots\,>\,\alpha_{I_{\ell-1}}$.

At $x_2$, in a similar fashion, we have a torsion submodule of dimension $t^2\,=\, I_1$, with a 
flag on $E/$torsion at $x_2$ of dimension $I_\ell-I_{\ell-1}+t^2,\, I_\ell-I_{\ell-2}+ 
t^2,\,\cdots$, corresponding to the shifted weights $-\alpha_{I_1}+1\,=\, 
-\alpha_{I_\ell}\,>\,-\alpha_{I_{\ell-1}} \,>\,\cdots\,>\,-\alpha_{I_2 }$. Note that this 
procedure preserves the symmetry of the strata. If we are in the general case, with no 
symmetry in the strata imposed between $x_1$ and $x_2$, then the situation for $x_1$ is 
simply reproduced at $x_2$.

The strata with $k\,=\,1$ in the moduli space will then correspond to vector bundles with a shifted degree, shifted 
weights, and induced flags. It is as if they were obtained from the degree zero case by a Hecke transform, except 
that the torsion ``mechanism'' obviates the need to have a subspace along which to take the transform. The ${\rm 
SL}(n,\C)$ structure vanishes at the $x_i$ with multiplicity $t^i$, but if one divides by the $t^i$-th power of the 
coordinate vanishing at $x_i$, one still obtains a top form at the points $x_i$.

The results of \cite{HuJS2} give us a moduli space $\mathcal C$ of framed parabolic {\it 
sheaves} over $\widetilde X_0$; the sheaves are torsion-free for weights 
$\alpha_1\,<\,\alpha_n+1$, but are allowed to acquire torsion (indeed forced to, if they 
correspond to a closed orbit in the GIT set-up) if we have $\alpha_1 \,=\, \alpha_n+1$ and 
$\beta^i_n\,=\, 0$.

\begin{theorem}[{\cite{HuJS2}}]\label{thm3.1}
There is a space $\mathcal C$ of semistable framed parabolic sheaves over $\widetilde X_0$, diffeomorphic to $C$ of \eqref{EC} under 
the Narasimhan-Seshadri correspondence. The elements of $\mathcal C$ are strings $({\mathcal E},\, \beta^i_j,\, i\,=\, 1,\,2,\, j\,= 
\,1,\,\cdots,\,n)$; there is a subdivision of $\mathcal C$ into strata $\mathcal C(I^1, k^1; I^2, k^2)$ according to the vanishing 
pattern of the $\beta^i_j$ such that the following hold:
\begin{itemize}
\item $\mathcal E$ is a coherent sheaf of rank $n$ over $X_0$, locally free if $k^1\,= \,k^2 \,=\, 0$, and in general having 
torsion over $x_1,\,x_2$ only; for a closed orbit in the GIT parametrization the torsion is isomorphic to a 
skyscraper sheaf $\C_{x_1}^{t^1} \oplus \C_{x_2}^{t^2} $ for integers $t_1\,=\, I^1_{\ell^1}- I^1_{\ell^1-1}$ and $t^2\,=\, 
I^2_{\ell^2}- I^2_{\ell^2-1}$. The sheaf $\mathcal E$ is equipped with an ${\rm SL}(n, {\mathbb C})$ structure in 
the sense of \cite{HuJS2}, that is, an identification $\bigwedge^{\rm max}({\mathcal E}/{\rm torsion})\, \simeq\, 
{\mathcal O}_{X_0}(-T)$, with $T$ being the divisor $t^1x_1 + t^2x_2$. If the torsion
part vanishes, this is just a standard ${\rm SL}(n, {\mathbb C})$ structure.

\item $\beta^i_j$ are decomposable elements of $\bigwedge^j(\mathcal E\big\vert_{x_i})^*$; the 
elements $\beta^i_j$ for a fixed $i$ are assumed to be compatible (see \eqref{eg}), and $\beta^i_n$ are 
supposed to be compatible with the ${\rm SL}(n,{\mathbb C})$ structure. If there is torsion at 
$x_i$, this implies that $\beta^i_n\,=\,0$. More precisely, the terms satisfy 
$$\beta^i_{1}\,=\,\beta^i_{2}\,=\,\cdots \,=\,\beta^i_{t^i-1}\,=\, 0,$$ if the torsion has 
rank $t_i$; furthermore, if $s_1, \,s_2,\,\cdots,\, s_{t_i}$ generate the torsion at $x_i$, then 
$\beta^i_{t_i}(s_1,\, s_2,\, \cdots,\, s_{t_i}) \,\neq\, 0.$

\item Taking the annihilators of the $\beta^i_j$ defines flags at $x_i$, and the forms $\beta^i_j$, combined with 
the ${\rm SL}(n,\C)$ structure, define volume forms on the subquotients of the flag.
\end{itemize}
\end{theorem}

The diffeomorphism in Theorem \ref{thm3.1}
with $C$ implies that we then have a map $\mathcal C\,\longrightarrow\, \Delta\times \Delta$, with $\Delta$ 
being the simplex of weights. Again, the faces of $\Delta\times \Delta$ correspond to the pattern of vanishing of the 
$\beta^i_j$ and to the multiplicity of the torsion sheaf and so to the dimensions of the subquotients of the two flags 
obtained as the annihilators of the non-zero $\beta^i_j$; in addition, the faces with $\alpha^i_n \,=\, \alpha^i_1 -1$ 
encode the existence of torsion. Again, we are just interested in the strata with $k^1\,=\, k^2$ and index $I^1
\,=\,\widehat R(I^2)$ and will just keep the labelling of $x_1$ referring to the stratum as $\mathcal C_{(I^1,k^1)}$ or 
$C_{(I ,k )}$. For $C_{(I ,k) }$, the non-vanishing $\beta^i_j$ are, for $k\,=\,0$:
$$\beta^1_n,\, \beta^1_{n-I_1},\, \beta^1_{n- I_{2}},\, \cdots ;\ \
\beta^2_n,\,\beta^2_{I_{\ell-1}},\, \beta^2_{I_{\ell-2}}\,\cdots $$
and for $k\,=\,1$, the same set apart from $\beta^1_n, \,\beta^2_n$, which now vanish. The 
condition for stability is expressed on each $\mathcal C_{I, k}$ by saying that a framed 
bundle is stable (respectively, semistable) if and only if it is stable (respectively, 
semistable) as a parabolic vector bundle for one of the weights in the corresponding stratum of 
$\Delta\times \Delta$.

Again, we set $\mathcal C^R \,=\, \bigsqcup_{I, k} \mathcal C_{I,k}$.

\subsection{On the singular curve $X_0$: glueing}

There is a natural action of two copies of the diagonal torus $T_\C$ of ${\rm SL}(n,\C)$, 
associated to each of the punctures: if $\mu_j,\ j\,=\, 1,\,\cdots ,\,n-1$, form the 
standard basis of characters for $T$, associated to the $j$-th exterior power of the basic 
representation, then $T_C\times T_C$ acts on $\beta_i^j$ by
$$(t_1,\, t_2)(\beta^1_j,\,\beta^2_j)
\,=\,\left(\mu_j(t_1)\beta^1_j,\, \mu_j(t_2)\beta^2_j\right).$$
If $t \,=\, diag(t^1,\,\cdots ,\,t^n)$, set $\nu_{s,s'}
\,=\, \mu_t/\mu_s \,=\, 
t^{s+1}\cdot\ldots \cdot t^{s'}$; if $\beta^i_{s'}\,=\, \gamma\wedge \beta^i_s$, then the action of $t_i$ on $\gamma$ is by 
$\nu_{s,s'}(t_i)$.

Our use of the reversal permutation acting on the weights means that we now consider the anti-diagonal torus $ T_\Delta 
\,\subset\, T\times T$ given as the subgroup of elements of the form $$((t^1,\,\cdots ,\,t^n),\, \,( t^n,\, \cdots,\,
t^1))\, .$$ On 
the stratum $ \mathcal C_{I,R(I)}$, where $I \,= \,(I_1, I_2,\cdots , I_\ell)$, the
group $ T_\Delta $ acts on the top exterior products 
of the successive quotients $E_1^{ I_s}/ E_1^{I_{s-1}}$ and $E_2^{I_{\ell-s}}/ E_2^{I_{\ell-s -1}} $, i.e., what 
corresponds to the top exterior powers of the eigenspaces of $\alpha_s$ at $x_1$ and the eigenspace of $-\alpha_s$ 
at $x_2$. The action is such that the composition of the trivialization at $x_1$ and the inverse of the 
trivialization at $x_2$ (twisted by $R$) is left invariant, so that we have a well defined identification. The 
algebraic quotients $$\widetilde {\mathcal M}_0 \,=\, \widetilde{\mathcal C}^R/\!\!/T_\Delta,\ 
\ \ \mathcal M_0\,=\, \mathcal 
C^R/\!\!/T_\Delta$$ then correspond to spaces of framed parabolic bundles or sheaves, with the top exterior powers of 
the subquotients of the flags identified at the two points. One has the identification, from \cite{HuJS2}, 
$$M_0 \,=\, {\mathcal M}_0\, . $$

It is important to note that what we obtain through this quotienting process are not --- even for the generic locus 
corresponding to distinct eigenvalues --- vector bundles on $X_0$. We instead have vector bundles on $\widetilde X_0$ (alternately, 
their direct images on $X_0$), with flags at both points $x_1,\, x_2$, and with identifications of the top exterior 
powers of the subquotients of the flags on both branches. In the symplectic case, where we have a scalar product, 
this identifies the fibers at the two singular points, at least when there are full flags. Here, in the holomorphic 
case, it does not.

\subsubsection{Bundles and their twists on the family $\X$}

If we consider bundles over our family $\X$ of curves, we have a bundle over the singular curve $X_0$. However, 
referring to the previous subsection, this is not what we want; we really would like to have sheaves over $\X$ whose 
restrictions are the (partially glued) direct images on $X_0$ of bundles on the blow-up $\widetilde X_0$ of the 
curve $X_0$. Dropping the partial glueing for the moment, it would suffice to take bundles tensored by the ideal 
sheaf of the origin in $\X$; we will, to make things a bit clearer, blow up, and consider bundles on the blow-up twisted 
by the negative of the exceptional divisor.

We 
therefore blow up $\X$ at the origin, to obtain $\widetilde \X$. The exceptional divisor is denoted by $D$. The lift of $X_0$ to 
this is a curve $\widetilde X'_0$, given over $\widetilde B$ as the union of the $x$-axis, the $y$-axis, and two times 
the exceptional divisor $D$: $$\widetilde X'_0 \,\,=\,\, \widetilde X_0 + 2D.$$

Now consider a bundle $\E$ over $\X$, of degree zero on the generic $X_t$, and lift it to $\widetilde \X$. The 
restriction of $\E$ to $D$ is trivial. Suppose that we twist $\E$ by
${\mathcal O}_{\X}(-D)$. This gives a vector bundle $\E(-D)$ whose 
restriction to $D$ is ${\SO}_D(1)^{\oplus n}$, and whose restriction to $\widetilde{X}_0$ has degree $-2n$; its 
restriction to the generic $X_t$ still has degree zero. Now note that the direct image of $\E(-D)$ to $\X$ (which is 
just the subsheaf of $\E$ of sections vanishing at the origin) has the behaviour that we want; the restriction to 
$X_0$ looks like the direct image of a bundle on $\widetilde{X}_0$.

\subsubsection{Isomonodromy and the Narasimhan-Seshadri theorem}

We now consider the deformations of vector bundles over $U$ given by isomonodromy. Deforming one 
bundle with a flat connection over a curve $X_p$ to neighboring curves isomonodromically 
gives a holomorphic family: the vector bundle, in a suitable cover, can be given by constant 
(hence holomorphic) transition functions, and these of course extend to neighborhoods. What 
it does not do, however, is give holomorphic isomorphisms, even local, from the moduli 
space $M_p$ of vector bundles over $X_p$ to the moduli space $M_{p'}$ associated to a neighboring 
curve; the passage to flat unitary connections destroys the holomorphic structure in any 
family (except when they are zero-dimensional) in the fiber. It does however give a 
diffeomorphism from $M_p$ to $M_{p'}$, for $p,\,p'\,\neq\, 0$, and a map from $M_p$ to 
$M_0$, which is a diffeomorphism over a generic locus. In any case, we want isomonodromic 
deformation of the Narasimhan-Seshadri connections to map fibers to fibers over $U$. In 
other words, our putative moduli space $\M$ over the disk $U$ should have fibers over $t$ 
the moduli space $\mathcal M_t$ of bundles over $X_t$ and our moduli space $\mathcal M_0$ 
over $t\,=\,0$.

{}From the variational point of view developed by Donaldson \cite{Do}, 
Atiyah--Bott \cite{AB} and other authors, this is quite sensible: we have a space of 
partial connections over $U$, whose fiber over $t\,\neq\, 0$ is the space of connections on 
$X_t$, and whose fiber over $X_0$ is the space of connections with a parabolic singularity 
at the points $x_1,\, x_2$. The minima of the Yang-Mills functional on all these fibers are 
the moduli spaces $\mathcal M_t$ and $\mathcal M_0$. The isomonodromy deformation of the 
flat connection already places one in this minimal locus; more generally, a small 
continuous deformation of the isomonodromic deformation, for example to a small holomorphic 
deformation, should retract to the minimal locus, adapting the arguments of Daskalopoulos 
\cite{Da}. More generally, with a fibrewise definition of stability, we should have, 
putatively:

\begin{conjecture}
Suppose that the bundle $\widetilde \E$ over $\X$ is $\vec\alpha$ polystable; then there is a complex gauge 
transformation ${\mathbf A}$ continuous at the origin, such that the Chern connection of the transformed metric
${\mathbf A}H_0{\mathbf A}^*$ gives flat (partial) Chern connections along each $X_t$.
\end{conjecture}

We leave discussion of this to another venue.

\subsubsection{Isomonodromy and parabolic structures}

Meanwhile, in $\mathcal M_t$, by the Narasimhan--Seshadri theorem, the bundle is determining a representation of the fundamental group 
and so a holonomy along the vanishing cycle $\gamma_t$. As shown above, the limit for our isomonodromic family of the holonomy along 
the cycles $\gamma_t,\ t\,\neq\, 0$, gives weights $\vec\alpha,\, -\vec\alpha$ for the parabolic structures along the curve $\widetilde 
X_0$. More generally, applying the Narasimhan--Seshadri theorem to a vector bundle over the family of curves tells us that unlike the 
case of an isolated curve, the holomorphic structure of the bundle on the family $\X$ should already be determining the weights in the 
limit.

One way of making more systematic sense of this is to re-examine the meaning of a parabolic 
structure. Referring to the analytic work of Biquard \cite{Bi}, one has the following three model 
trivializations of a bundle that come into play when one considers a parabolic structure at the 
origin of a vector bundle over a punctured disk, with coordinate $z\,=\,r\exp({\sqrt{-1}\theta})$:
\begin{itemize} 
\item A unitary trivialization, which is not holomorphic; there is a unitary 
connection given by $d+ \sqrt{-1}\vec\alpha {d\theta}$ with respect to this trivialization,
in which the $\overline\partial$ operator is given by 
$(\frac {\partial}{\partial \overline z} - \frac{1}{2}\frac{\vec\alpha}{\overline z})d\overline z$.

\item A holomorphic, univalent trivialization in which the connection is given by $d + \frac{\vec\alpha}{z}dz$. 
This trivialization is related to the first one by a gauge transformation $r^{-\alpha}$ (so that the new basis behaves like $r^{\alpha}$).

\item A holomorphic, flat but multi-valued trivialization, related to the second one by a gauge transformation 
$z^{-\alpha}$.
\end{itemize}
 
We begin with a single curve:

\begin{definition}
A {\it parabolic structure} on a holomorphic vector bundle $E$ over a curve $Y$ at a point $p$, corresponding to $z\,=\,0$, with 
weights $\vec\alpha$ is the expression of the bundle over a disk containing $p$ as the sheaf of holomorphic sections on 
the punctured disk of a $\overline\partial$ operator on a unitary bundle $F$, conjugate to $(\frac {\partial}{\partial 
\overline z} - \frac{\vec\alpha}{2\overline z} )d\overline z$.

Such a sheaf has as invariants a standard flag $0\,\subset\, E_{m_1}\,\subset\, E_{m_2}\,\subset\,\cdots$ at the 
origin, corresponding to the decay rates of sections given by the different $\alpha_i$. Indeed the automorphisms of 
such a bundle preserving this operator are holomorphic over $E$, and to be continuous on $F$ must be block upper 
triangular at the origin. (See \cite{MS}, \cite{Bi}.)

Alternately, the parabolic structure can be given as the datum of a holomorphic connection, conjugate on the disc to 
$(\frac{\partial}{\partial z} + \frac{\vec\alpha}{z})dz$. The flat sections of $Aut(E)$ for this connection which 
remain finite at the origin determine in the gauge bundle $Aut(E)$ a parabolic subgroup $P$ of $Aut(E)$ at the origin. One considers the 
connections modulo the equivalence by gauge transformations on the disc taking values in $P$ at the origin; again the 
invariants are flags.
\end{definition}

In short, parabolic structures are given by growth rates of univalent holomorphic sections at the singular point, in a 
unitary trivialization. (See \cite{MS}.)

Note that the vector bundle $E$ has a natural Hecke transform to a bundle $\widehat E$, a subsheaf of $E$, with a 
corresponding shift of the weights to a $\widehat \alpha$ given by shifting the weights equal to $\alpha_n$ up
by $1$, to $\alpha_n+1$. This is particularly useful in our case when $\alpha_n\,=\, \alpha_1-1$.

Now observe that the flat sections $z^{-\alpha}$ of the parabolic structure in its holomorphic version can be thought 
of as being given locally as the orbits of a $\C^*$ action, acting with weight one on the base, and with weight 
$-\vec\alpha$ on the fiber, with each $\alpha_i$ corresponding to an action on the tautological lines of the flag 
manifold on the fiber; we see that the data for $\alpha$ varying (i.e., framed parabolic structures) is encoded as 
follows:

\begin{definition}
As in \cite{HuJS2}, a {\it framed parabolic structure} on a bundle $E$ at a point $p$, corresponding to $z=0$, with 
weights $\vec\alpha$ is the expression of the bundle over a disk containing $p$ as the sheaf of holomorphic sections 
on the punctured disk of a $\overline\partial$ operator on a unitary bundle $F$ conjugate to 
$(\frac{\partial}{\partial \overline z} + \frac{\vec\alpha}{\overline z} )d\overline z$, with in addition a 
trivialization of the top exterior powers of the subquotients $\bigwedge^{\max}(E_{I_j}/E_{I_{j-1}})$, such that the 
norm of the trivialization is (a fixed multiple of) $\alpha_{I_j}- \alpha_{I_{j-1}}$. When one is dealing in 
addition with an ${\rm SL}(n,{\mathbb C})$-structure, one asks that the product of the framings give the volume form at $p$. 
\end{definition}

When the weights are at the boundary $\alpha_1 \,=\, \alpha_n+1$, we again have to modify 
our construction a bit. We can still have a structure expressing $E$ as the holomorphic 
sections of a $\overline\partial$ operator $(\frac {\partial}{\partial \overline z} + 
\frac{ -\vec\alpha}{2\overline z} )d\overline z$, but now we take the Hecke transformed 
$\widehat E$, and also take the Hecke transformed flag. The framing is now a trivialization 
of the top exterior powers of the subquotients for $\widehat E$.

\begin{remark} Fix the standard basis $e_1,\, e_2,\, \cdots,\, e_n$ of $\C^n$, and consider the forms $e_1^*,\, e_1^*\wedge 
e_2^*,\, \cdots, e_1^*\wedge\cdots\wedge e_n^*$ defining the standard flag. For each multiplicity stratum $I \,=\, 
(I_1,\,\cdots,\,I_\ell)$, the possible corresponding ``quotients''
$$\gamma_1\, =\, c_1 e_1^*\wedge\cdots\wedge e_{I_1}^*,\, \gamma_2 \,=\, 
c_2e_{I_1+1}^*\wedge\cdots \wedge e_{I_2}^*,\,\cdots$$ trivializing top exterior powers of the subquotients, give a subvariety 
isomorphic to $(\C^*)^{\ell-1}$; taking the union over all multiplicity strata, gives a description of the projective 
space $\bP^{n-1}$ as a toric variety, with moment polytope the polytope of weights $\vec\alpha$. It is this 
relationship between a holomorphic variety and the weights $\vec\alpha$ that serves as model for the relationship 
between the moduli space, and the exponents $\alpha_i$ of the parabolic structure. \end{remark}

We go back now to our family $\X$ of curves, intersected with $B$, and define a doubly parabolic structure for a 
bundle $\widetilde \E$ over $\widetilde \X$ over the origin $p$; we suppose that the first Chern class of $\E$ is zero. We note again that our purpose here is to build in to the family a singular $\bar\partial$ operator, with its singularity at $X_0$, so that the parabolic structure at $X_0$ appears naturally as a limit of the smooth structures on $X_t$.

\begin{definition}
A {\it uniform doubly parabolic structure} on $\widetilde \E$ on $\widetilde B$, with weights $\vec\alpha^1,\, 
\vec\alpha^2$ is the expression of the bundle over $\widetilde B\,\subset\,\widetilde \X $ as the sheaf of holomorphic 
sections on $\widetilde B $ of a $\overline\partial$ operator on a unitary bundle $F$, conjugate (in the blown up 
coordinates of \eqref{blowup1}, \eqref{blowup2}) to ${\overline \partial} + \frac{\vec\alpha^1}{2\overline {\widetilde 
y}} d\overline {\widetilde y}$ near $\widetilde x\,=\,0$ (near $x_1$) and $ {\overline \partial} + \frac{ \vec\alpha^2 
}{2\overline {\widehat x} }\,d\overline {\widehat x}$ near $\widehat y \,=\, 0$ (near $x_2$), so that the poles of the 
$\overline\partial$-operator are over $\widetilde X_0\cap \widetilde B$.

Alternatively, the uniform doubly parabolic structure can be given on $\widetilde \E$ by a holomorphic flat connection 
with a pole along $\widetilde X_0$, conjugate to $ { \partial} - \frac{\vec\alpha^1}{ 2{\widetilde y}} d {\widetilde y}$ 
near $\widetilde x\,=\,0$ and ${\partial} - \frac{\vec\alpha^2}{ {\widehat x} }d{\widehat x}$ near $\widehat y\,=\, 0$.
\end{definition}

We note that while this expression for the connection has a pole along $\widetilde X_0$, if 
one takes a limit in $t$ along the curves $X_t$, the limit on $\widetilde X_0$ is a 
connection with a pole only at $x_1,\, x_2$, with residues $\alpha^1, \,\alpha^2$ respectively.

We are given the weights of the structure $$\alpha^1_1\,\geq\, \alpha^1_2\,\geq\,\cdots\,
\geq\, \alpha^1_n
\,\geq\, \alpha^1_1 -1,\,\,\, 
\sum_{i=1}^n\alpha^1_i \,=\, 0$$ at $x_1$ and $$\alpha^2_1\,\geq\, \alpha^2_2\,\geq\,\cdots\,
\alpha_n^2\,\geq\,\alpha_1+1,\, 
\,\,\sum_{i=1}^n\alpha_i^2 \,=\, 0$$
at $x_2$, lying within the set $\Delta$. Now group the ones that are equal together, as above, 
so that $I^1_1,\, \cdots,\, I^1_\ell$ are the indices for $x_1$ for which $\alpha^1_{I_j}\,>\,\alpha^1_{I_j+1}$, and 
similarly for $x_2$.

This gives as invariant flags of subbundles $E^1_{I^1_j}$ along the $x_1$ component of $ \widetilde X_0\cap 
\widetilde B$, and $E^2_{I^2_j}$ along the $x_2$ component. This is more than what we want: we just want the structure 
over the two points $x_1,\,x_2$, the intersections of the two branches of $ \widetilde X_0$ with $D$. We therefore 
define:

\begin{definition} A {\it doubly parabolic structure} on $\widetilde \E$ at $D\,\subset\, \widetilde B$, with weights
$\vec\alpha^1,\, \vec\alpha^2$ is the expression of the bundle over ${\widetilde B}\,
\subset\, \widetilde{\X}$ as an equivalence class 
of uniform doubly parabolic structures under the action of holomorphic maps of a neighborhood of $\widetilde 
X_0\cap\widetilde B$ into ${\rm GL}(n, \C)$ which are the identity at $x_1,\, x_2$.
\end{definition}

This structure has as invariants a pair of flags $E^1_{I^1_j}$ at $x_1$ and $E^2_{I^2_j}$ at $x_2$ corresponding to 
the flags of parabolic structure as one approaches the origin along the $x$ and $y$ axis, respectively. We note also 
that restricting the connection to the divisor $D$, instead of $\widetilde x\,=\,0$, reverses the sign of the 
residues of the connection, with weights $-\vec\alpha^1/2$. The same holds at $x_2$. We nevertheless use weights 
$\vec\alpha^1,\, \vec\alpha^2$ on all of the curve $X_0'$ when discussing parabolic structures and the ensuing 
stability.
 
Now we put in a framing at the two points; note now for the GIT quotient problem that instead of specific weights, 
we are simply on a stratum defined by the choice of two faces of the simplex $\Delta$, i.e., by the multiplicity 
patterns $I^1_j,\, I^2_j$ of the weights.

\begin{definition} A {\it framed doubly parabolic structure} on a bundle $\widetilde \E$ over $\widetilde X$ at the 
point $p$, is a doubly parabolic structure, with in addition trivializations of the top exterior powers of the 
subquotients $E^1_{I^1_j}/E^1_{I^1_j-1}$ and $E^2_{I^2_j}/E^2_{I^2_j-1}$, again such that the norm of the trivializations are (a fixed 
multiple of) $\alpha^i_{I_j}- \alpha^i_{I_{j-1}}$. For an ${\rm SL}(n, \C)$ structure, we ask in 
addition that the product of the trivializations be compatible with the ${\rm SL}(n, \C)$ structure of the 
bundle.\end{definition}

The flag, and the trivializations, are encoded at $x_1$ in decomposable elements $\widetilde{\beta}^1_{I_j}$ of 
$\bigwedge^{n-I_j}(E)^*$ whose annihilator is the $I_j$--dimensional vector space $E^1_{I_j}$ of the flag. We have the top 
form $\widetilde{\beta}^+_{n}$, which is just the ${\rm SL}(n,\C)$ structure, and has as annihilator the zero vector space $E^1_0$.
The fact that we are dealing with a flag tells us that there are decomposable elements $\gamma^1_j$ of
$\bigwedge^{I^1_j- I^1_{j-1}}(E^*)$ with
$$\widetilde{\beta}^1_{I_{j-1}} \,\,=\,\, \gamma^1_j\wedge \widetilde{\beta}^1_{I_{j}}\,. $$
We write this as $\gamma^1_j \,= \,\widetilde{\beta}^1_{I_{j-1}}/\widetilde{\beta}^1_{I_{j}} $; now $ \gamma^1_j$
defines a volume form on $E^1_{I_j}/E^1_{I_{j-1}}$. We have similar quantities
$\widetilde{\beta}^2_{R(I)_j}$ at $x_2$ encoding the other parabolic structure. As noted above,
the data $\vec\alpha^1,\, \vec\alpha^2$ get encoded in a moment map applied to the elements $\gamma^i_j $.

Next, we impose some symmetry. We restrict to strata with $\alpha^2\,=\,-R(\alpha_1)$, i.e., with $\alpha^2_j \,=\, 
-\alpha^1_{n-j+1}$; the two points then have the same multiplicity pattern, but in inverse order. Corresponding to 
these we have flags $E^1_{I^1_j}$,\, $E^2_{n-I^1_j}$ at the two points, and in the framed case, decomposable elements 
$\widetilde{\beta}^1_{I_j}$,\, $\widetilde{\beta}^2_{n-I_j}$.

We also have a natural action of the torus of ${\rm SL}(n,{\mathbb C})$ on the framings at each of the points 
$x_1,\, x_2$, defined uniformly over the strata. Consider the ``antidiagonal'' torus $T$ in the product, consisting 
of elements $(diag (\mu_1,\,\cdots,\,\mu_n),\,\,diag (\mu_n,\,\cdots ,\,\mu_1))$; quotienting by this torus gives an 
identification of the framings at $x_1$,\, $x_2$.

\begin{definition}
A {\it symmetric doubly parabolic structure} on a bundle $\widetilde \E$ over $\widetilde 
X$ at the point $p$ is a doubly parabolic structure with symmetric weights satisfying 
$\alpha^2\,=\, -R(\alpha_1)$. A {\it framed symmetric doubly parabolic structure} comes in 
addition with a trivialization of the top exterior power of each of the subquotients 
$E^1_{I^1_j}/E^1_{I^1_j-1}$ and $E^2_{I^2_j}/E^2_{I^2_j-1}$. For an ${\rm SL}(n, \C)$ 
structure, we ask in addition that the product of the trivializations be compatible with 
the ${\rm SL}(n, \C)$ structure of the vector bundle.
\end{definition}

Note that we have in addition two copies of the torus $T$ acting naturally on the 
trivializations of the top exterior powers of $E^1_{I^1_j}/E^1_{I^1_j-1}$ and 
$E^2_{I^2_j}/E^2_{I^2_j-1}$; this action extends well to the cases of the bundles acquiring 
torsion. Quotienting by the antidiagonal torus, as above, amounts to identifying these 
quotients:

\begin{definition} A {\it reduced framed doubly parabolic structure} on a bundle $\widetilde \E$ over $\widetilde X$ 
at the point $p$ is a doubly parabolic structure, with weights lying in $\Delta^R$, and in addition identification
of the top exterior power of each of the subquotients $E^1_{I_j}/E^1_{I_{j-1}}$ and $E^2_{n - I_{j+1}}/E^2_{n-I_{j }}$. For 
an ${\rm SL}(n, \C)$ structure, we again ask in addition that the identifications be compatible with the global 
trivialization of the top exterior power of the vector bundle.
\end{definition}

We note that the action of $T$, as it changes the norms of $\widetilde{\beta}^1_{I_j},\,\, \widetilde{\beta}^2_{R(I)_j}$, modifies the 
$\vec\alpha^1,\, \vec\alpha^2$, and does not (as it should not) preserve the equality $\vec\alpha^2 \,=\, -R(\vec\alpha^1)$ of the 
symplectic reduction. It does however respect the stratification, and there will be elements in the orbit for which the equality does 
hold; this is necessary for the equivalence of the symplectic and holomorphic quotients.

It will be the moduli of these objects that we want to consider at $t\,=\,0$.

\subsection{Stability}

We will restrict to the symmetric case. 

\begin{definition}
We will say that a symmetric doubly parabolic bundle $\widetilde \E$ is $\vec\alpha$ stable (respectively,
semistable) if
\begin{itemize}
\item its restrictions $\widetilde E_t$ are stable (respectively, semistable) in the usual sense as bundles over each
$X_t$,\, $t\,\neq\, 0$,

\item and the restriction $\widetilde{E}_0 $ to $X'_0$ is 
$\vec\alpha$ stable (respectively, semistable) as a doubly parabolic bundle over $\widetilde X'_0$.
\end{itemize}
\end{definition}

Note that this pointwise definition of stability is compatible with the stability for 
families considered by Seshadri in \cite{Se} (pp. 252 ff.). The above definition is 
natural. Indeed, the question of whether an object is semistable is related to whether its 
orbit under a linearization is bounded away from the origin; the orbit will be bounded away 
from the origin in a (compact) family if and only if it is bounded away from the origin at 
every point of the family.

\begin{definition}
We will say that a framed symmetric doubly parabolic bundle $\widetilde \E$ with multiplicity pattern $I,\,k$ is
stable (respectively, semistable) if
\begin{itemize}
\item its restrictions $\widetilde E_t$ are stable (respectively, semistable) in the
usual sense as bundles over each $X_t,\, t\,\neq\, 0$, and

\item the restriction
$\widetilde{E}_0 $ of $\widetilde \E $ is $\vec\alpha$ stable (respectively, semistable) as a reduced doubly parabolic
bundle over $\widetilde{X}'_0$, for one of the weights $\vec\alpha$ in the stratum $\Delta^{(I, k)}\,\in\, \Delta^R$.
\end{itemize}
\end{definition}

The curve $\widetilde X'_0$ has one component that is a (double) projective line $2D$. On 
the line $D$, the sheaf $\widetilde{E}_0$ (modulo torsion) will split as a direct sum of 
line bundles of the form ${\SO}_D(j)$.

\begin{lemma}\label{lod}
Let $\vec\alpha^2 \,=\, -R(\vec\alpha^1)$. If $\alpha_1\,<\, \alpha_n+1$, then 
$\widetilde{E}_0$ is trivial over $D$; if we are at a closed point of the prequotiented GIT 
space when $\alpha_1\,=\, \alpha_n+1$, then $\widetilde{E}_0$ (modulo torsion) is a bundle $ 
{\SO}_D(-1)^k\oplus {\SO}_D^{n-k}$ on $D$. Thus the rank $k$ of the torsion is encoded by the type 
of the bundle.
\end{lemma}

\begin{proof}
One checks that it is necessary to have the semistability simply on the restriction of the 
bundle to $D\,=\,\bP^1$. If ${\SO}_D(j)$ is a subbundle on $\bP^1$, the stability forces the 
corresponding weights $\alpha_s ,\, -\alpha_t$ at $x_1,\, x_2$ of the parabolic structure 
to satisfy $j\,\leq\, (\alpha_s -\alpha_t)$. If $\alpha_1\,<\, \alpha_n+1$, then $j\,\leq\, 
0$. For the case $\alpha_1\,=\, \alpha_n+1$, quotienting out the torsion ``shifts" the 
weight $\alpha_n$ at $x^1$ up by one, and the weight $-\alpha_1$ at $x_2$ up by one. The 
upper bound for $j$ is then still a difference $(\alpha_s -\alpha_t)$ for the shifted 
weights, and so $j\,<\,0$; on the other hand if $\SO(j)$ is a quotient bundle, then we get 
a bound $j\,\geq\, -1$.
\end{proof}
 
\subsection{A fibered GIT problem}

We can now treat the fibered GIT problem of classifying our bundles $\widetilde \E$. As defined above, we consider 
over $\widetilde \X$ reduced doubly parabolic ${\rm SL}(n,\C)$ vector bundles $\widetilde \E$ satisfying a fibrewise 
semistability condition, with ordinary bundle semistability over $X_t,\,\, t\,\neq\, 0$, and (symmetric) framed doubly 
parabolic stability over the restriction of the bundle to $\widetilde{X}_0$.

The aim will be to build a moduli space as a fibrewise GIT quotient of vector spaces derived from the space of sections 
of $\widetilde \E$ (or rather some twist of it). We first see that this can be done uniformly in $t$.

To describe these sections, let $\mathbf{C}$ be a divisor over $U$, intersecting each curve $X_t$, including $X_0$, with 
multiplicity $N_{\mathbf{C}}.$
Now we want to compute the dimensions of $$H^0(X_t,\, \widetilde{\E}(\mathbf{C}-D))$$ for $t\,\neq\, 0$, and $H^0(X_0', 
\,\widetilde{\E}(\mathbf{C}-D))$, assuming that we are in a stable range where the relevant $H^1$'s vanish.

\begin{lemma}\label{lod2}
Let $c_1(\widetilde{\E})\, =\, 0$. Assume that $ H^1(X_t,\, \widetilde 
\E(\mathbf{C}-D\,))\, =\, 0$ for all $t$. Then for all $t$ (noting that
$\widetilde\E(\mathbf{C}-D) =\widetilde\E(\mathbf{C})$ over $t\,\neq\, 
0$), $$h^0(X_t,\,\, \widetilde{\E}(\mathbf{C}))\,=\, 
h^0(X_0' ,\,\, \widetilde{\E}(\mathbf{C}-D))\,= \,n(1-g) + n (N_{\mathbf{C}})\, .$$ 
\end{lemma}

\begin{proof}
For $t\,\neq\, 0$, Riemann--Roch gives that $$\dim H^0(X_t,\, \widetilde{\E}(\mathbf{C}))\,=\,
n(1-g) + n (N_{\mathbf{C}}).$$ For $t
\,=\,0$, assume first that 
$\widetilde \E$ is the lift of a bundle $F$ from $\X$. Then $H^0(2D,\, \widetilde{\E}(-D))$ is the 
space of sections of $F$ on the second formal neighborhood of the origin, but vanishing at the origin. This gives a 
space which is $5n$-dimensional. We then consider sections on $X_0$: the restriction of $\widetilde{\E}$ has
degree $0$, and so that of $\widetilde{\E}(\mathbf{C}-D)$ is $n N_{\mathbf{C}} -2n$; the space of 
sections of $\widetilde{\E}(\mathbf{C}-D)$ is then $(n(1-(g-1)) -2n +nN_{\mathbf{C}})$-dimensional. Sections over
$X_0\cup 2D$ require equality over the intersection. 
This gives $4n$ constraints when $N_{\mathbf{C}}$ is large, and so again we get a space of dimension
$$n (1-(g-1))+ nN_{\mathbf{C}} + 5n-2n-4n\, =\, n(1-g) + n(N_{\mathbf{C}})\, .$$
Now, even if $\widetilde \E$ is not a lift, as we are dealing with an Euler characteristic, and $H^1$ vanishes, the result holds for any deformation, and so in general.
\end{proof} 

We now define families $\mathcal F$ of sheaves $\widetilde\E$ of degree $m$ on $\widetilde\X$, trivial on $D$, by 
asking that for all subsheaves $\widetilde \E'$ of positive rank $k'$ and degree $m'$ with the same torsion as $\widetilde\E$, that
$$\frac{m'}{k'} \,<\, 2m.$$
The families $\mathcal F$ contain fairly easily the sheaves satisfying our eventual notion of stability. For the moment:

\begin{lemma} 
These families are bounded; that is there is a twist by a positive divisor $\mathbf{C}$ as above such that the sheaves 
$\widetilde\E(\mathbf{C}-D)$ are generated by local sections and that the first cohomology $H^1(X_t,\,\E(\mathbf{C}-D))\, =\, 0$, for $t\in D$.
\end{lemma}

For $X_0$, this is Lemma 3.2 of \cite{HuJS2} and Lemmata \ref{lod2} and \ref{lod} above on the type of the bundle 
over $D$; for the $X_t$, the family $\mathcal F$ is defined by a much weaker notion than stability.
 
We recall the encoding of these sheaves by algebraic data. We do this following Bhosle 
\cite{Bh}, with a small twist. Recall that we are dealing with ${\rm SL}(n, \C)$ bundles, 
so with a fixed determinant. This gives a homomorphism $\bigwedge^n(\widetilde\E) 
\,\longrightarrow\, \SO$, and so a homomorphism $\bigwedge^n(\widetilde\E(\mathbf{C})) 
\,\longrightarrow\, \SO(n\mathbf{C})$, and also $\bigwedge^{n-1}(\widetilde\E(\mathbf{C}))\wedge 
\widetilde\E(\mathbf{C}-D) \, \longrightarrow\, \SO(n\mathbf{C} -D)$. Taking a direct image 
to $\X$ gives a homomorphism
$$\bigwedge\nolimits^{n-1}\left( \E(\mathbf{C})\right)\bigwedge
(\E(\mathbf{C})\otimes \SI_0)\,\longrightarrow\,\SO(n\mathbf{C})\otimes \SI_0.$$
Here $\SI_0$ is the ideal sheaf of the origin. This basically frees up the map at the origin so that it can take different values along the two branches of $X_0$.

Now project further and take a zero-th direct image to $U$, letting $V_0\,\subset\, V$ denote the direct images of $ 
\E(\mathbf{C}-D)\otimes \SI_0$ and $\E(\mathbf{C})$ over $U$ (note that the sections over the pre-images for
$V_0$ have constant rank); 
both sheaves are locally free. If $L$ is the direct image of $\SO(n\mathbf{C} -D)$, then we have a determinant map on 
sections $$\beta_0\,\,:\,\,\bigwedge\nolimits^{N-1}( V)\bigwedge V_0\,\longrightarrow\, L.$$
Here $N$ denotes the rank of the direct image of $\widetilde{\E}(\mathbf{C}-D)$. The map $\beta_0$ encodes the
sheaf on $X_t,\,\, \widetilde X_0$, modulo its torsion: indeed, lifting $\beta_0$ back to $\widetilde \X$ and evaluating sections in $L$ gives a sheaf version of the determinant map
$$B_0\,\,:\,\, \bigwedge\nolimits^{N-1}(V)\wedge V_0\,\longrightarrow \,\SO(n\mathbf{C} -D) .$$
If one sets $Ann(B_0)$ to be the subsheaf of sections $s_1$ of $V$ such that
$$B_0(s_1,\, s_2,\,\cdots,\,s_n)\, =\, 0$$ for all other sections, 
the quotient $V/Ann(B_0)$ is just $\widetilde{\E}(\mathbf{C})$, modulo its torsion. We then
recuperate, modulo torsion, the evaluation map of sections
$$ev\,\,:\,\, V\,\longrightarrow\, \widetilde{\E}(\mathbf{C})\, .$$

Note that the twist by $\SO(-D)\big\vert_D\,=\, \SO(1)$ over $D$, so that $\E(-D)$ restricted to $D$ is generically 
a sum of line bundles $\SO(1))$ just ``frees'' in the direct image the sections of $\widetilde \E$ over the two 
branches of the curve $\widetilde X_0$, allowing them to vary independently. The quotient construction gives sheaves 
over $\widetilde X_0$, not $\widetilde X_0' \,= \,\widetilde X_0 +2D$.
 
To encode the framed double parabolic structure, we proceed as in \cite{HuJS2}. The elements $\widetilde{\beta}^1_k
\,\in\,\bigwedge^{k}(\widehat{\E}\big\vert_{x_1})^*,\,\, \widetilde\beta^2_k \,\in\,
\bigwedge^{k}(\widehat \E\big\vert_{x_2})^*$,
become, using the 
evaluation, elements $$\beta^1_k \,\in \,\bigwedge\nolimits^{k}(\mathcal O^{\oplus N}\big\vert_{x_1}),\ \
\beta^2_k\,\in\, \bigwedge\nolimits^{k}(\mathcal 
O^{\oplus N}\big\vert_{x_2})\, .$$ As noted above, the elements at $x_1$ all 
have to be compatible; they should also be compatible with $B_0(x_1)$; the same holds at $x_2$. Our definition becomes, in the GIT context:
 
\begin{definition} Under the standard linearization (see \cite{HuJS2} for $t\,=\,0$), the elements $\beta_0,\,
\beta^1_k,\, \beta^2_k$ define a semistable framed doubly parabolic structure for ${\rm SL}(n,{\mathbb C})$ over $U$ if and only if
\begin{itemize}
\item they define a semistable holomorphic bundle over $t\,\in\, U$ for $t\,\neq\, 0$, and 

\item they define a semistable framed doubly parabolic structure on $\widetilde X_0$ over $0\,\in\, U$.
\end{itemize}
\end{definition}

\begin{theorem}
There is a family $\widetilde \M$ of moduli spaces over $U$ of symmetric framed doubly parabolic bundles, with fiber 
at $t\,\neq\, 0$ the moduli space $M_t$ of ${\rm SL}(n,\C)$-bundles over $X_t$ and, at $t\,=\,0$, the moduli space $M_0$ 
of framed symmetric doubly parabolic bundles over $\widetilde X_0$.
\end{theorem}
 
\begin{proof}
We first note that the space $\widetilde \M$ has already been constructed above in terms of representations into
${\rm SU}(n)$; see 
(\ref{MM-sympl}). Also, isomonodromic deformation gives a Cartan connection on the fibration of $\widetilde \M$ over $U$, giving 
us a good way of obtaining diffeomorphisms from fiber to fiber away from $t\,=\,0$, as well as a way to map fibers at 
$t\,\neq\, 0$ to the central fiber. 

Furthermore, the Narasimhan Seshadri theorem, at $t\,\neq\, 0$, and Proposition 3.2 of 
\cite{HuJS2} at $t\,=\,0$ give us analytic structures on the fibers of $\widetilde\M\,\longrightarrow\, U$, as spaces
of vector bundles on the curves $X_t$, and of 
spaces of symmetric framed parabolic bundles on $\widetilde X_0$. We thus have a smooth (i.e., smooth away from the 
singularities of the fibers) family of algebraic varieties above the disk $U$. The question is what is the holomorphic 
structure transverse to these fibers, or what are the holomorphic sections of $\widetilde\M\,\longrightarrow\, U$.
 
Isomonodromic deformation already gives us the holomorphic structure in the transverse direction; more generally, 
(noting that one way of seeing an isomonodromic deformation is to have it induced by an ambient connection on the 
surface $\widetilde \X$) we will have that any holomorphic family of bundles on the surface $\widetilde \X$ with a 
framed doubly parabolic structure will give us a holomorphic section of $\widetilde \M$.
\end{proof}

All that remains now is to quotient out the natural action of the anti-diagonal torus $T\,=\,
(\C^*)^{n-1}$ of ${\rm SL}(n,\C)^2$ to get a moduli space whose fiber at $t\,=\,0$ will be the 
{\it reduced} symmetric framed doubly parabolic bundles on $X_0$:
\begin{equation} \M \,=\, \widetilde \M /\!\!/ T\, . \end{equation}
This is the degeneration that we want. We note that the fiber over $0$ is a quotient of a 
moduli space associated to $\widetilde X_0$.

\section{Multiple nodes}

The above construction obviously iterates; one can take a family $X_{t_1, t_2,\cdots, t_s}$ of curves over 
$U_1\times U_2\times\cdots \times U_s$ degenerating to a curve $X_{0,\cdots,0}$ with $s$ nodes, with the 
generic $X_{t_1, t_2,\cdots ,t_s}$ (i.e., for $({t_1,\, t_2,\,\cdots ,\,t_s})$ away from the coordinate axes) being 
smooth, and build above $U_1\times U_2\times\cdots \times U_s$ a family $\M$ of moduli spaces:

\begin{theorem}
There is a family $ \M$ of moduli spaces over $U_1\times U_2\times\cdots \times U_s$ of reduced symmetric framed doubly parabolic bundles, with fiber at generic $({t_1,\,t_2,\,\cdots ,\,t_s})$ the 
moduli space $M_{t_1, t_2,\cdots,t_s}$ of ${\rm SL}(n,\C)$-bundles over $X_{t_1,t_2,\cdots ,t_s}$ and, at
$(0,\,0,\,\cdots ,\,0)$, the moduli space $M_{(0,0,\cdots ,0)}$ of {\it reduced} symmetric framed doubly parabolic bundles on $X_{0,\cdots ,0}$. This space is a 
$(\C^*)^{s(n-1)}$ quotient of a space of $\widetilde M_{(0,0,\cdots ,0)}$ of symmetric framed doubly parabolic bundles, a moduli space over the desingularised curve $\widetilde X_{0,\cdots ,0}$.

Symplectically, $ \M$ is isomorphic to a family of symplectic varieties over $U_1\times U_2\times\cdots \times U_s$, whose fiber at $({t_1,\, t_2,\,\cdots ,\,t_s})$ generic is the moduli space $M_{t_1, t_2,\cdots ,t_s}$ of
${\rm SU}(n)$ representations of the fundamental group of $X_{t_1, t_2,\cdots ,t_s}$, and whose fiber at $ (0,\,0,\,\cdots,\,0)$ is a symplectic quotient of the family of imploded representations of the fundamental group of the punctured desingularised curve $\widetilde X_{0,\cdots,0}^*$ (punctured at the preimages of the nodes), where the quotient is by an action of $(S^1)^{s(n-1)}$ on the framings at the punctures.
\end{theorem}

Now let us take the nodal degeneration of a curve $X$ into a union $X_0$ of $2g-2$ trinions, whose $6g-6$ punctures are identified pairwise at $3g-3$ nodes. The above theorem will give a degeneration of the moduli space over $X$ into a glueing of moduli spaces $M^T_i$ associated to trinions, 
so that
\begin{align}
\widetilde M_{(0,\,0,\,\cdots,\,0)} &\,=\,( \prod_i M^T_i)/\!\!/_{GIT} (\C^*)^{(3g-3)(n-1)}
\nonumber\\
&=\,( \prod_i M^T_i)/\!\!/_{sympl} (S^1)^{(3g-3)(n-1)}.\nonumber
\end{align} 
Here the action of an individual $(\C^*)^{n-1}$ or $(S^1)^{n-1}$ is the antidiagonal action on the framings 
associated to a pair of punctures at a given node. We note that though we had imposed the condition of symmetry, 
this is automatic in our symplectic picture, under the symplectic quotient, being forced by the vanishing of the 
moment map. It is then automatic in the complex picture too, by the standard theorem relating symplectic and 
algebraic quotients; the non-symmetric strata are unstable.

For $n\,=\,2$, by the results of \cite{HuJ}, we have obtained a limit which is a toric variety, 
as the space associated to the trinion is just $\mathbb P^3(\mathbb C)$. On the other hand, 
for higher ranks, further degeneration is necessary, an issue we hope to address in an 
upcoming paper.

\section*{Acknowledgements}

We thank Peter Newstead for pointing out some references.

\end{document}